\documentclass[11pt]{article}

\usepackage{amsmath}
\usepackage{amsthm}
\usepackage{amsfonts}
\usepackage{amssymb}
\usepackage{amsmath,amscd}

\newcommand{\Z}{\mathbb{Z}}
\newcommand{\N}{\mathbb{N}}
\newcommand{\PP}{\mathbb{P}}
\newcommand{\RR}{\mathbb{R^{+}}}
\newcommand{\Bl}{\operatorname{Bl}}

\newcommand{\Pic}{\operatorname{Pic}}
\newcommand{\NE}{\operatorname{NE}}
\newcommand{\NNE}{\overline{\operatorname{NE}}}

\newcommand{\cont}{\operatorname{cont}}
\newcommand{\Exc}{\operatorname{Exc}}

\newcommand{\Supp}{\operatorname{Supp}}
\newcommand{\Ratcurve}{\operatorname{RatCurves}^{n}}
\newcommand{\OO}{{\cal O}}
\newcommand{\EE}{{\cal E}}

\newcommand{\wtilde}{\widetilde}

\newtheorem{theo}{Theorem}
\newtheorem{prop}{Proposition}
\newtheorem{lem}{Lemma}

\newtheorem{claim}{Claim}

\title
{On the minimal length of extremal rays for 
Fano 4-folds}
\author{Toru Tsukioka}

\begin{document}

\maketitle

\begin{abstract}
The minimum of intersection numbers of the anti-canonical 
divisor with rational curves on a Fano manifold is called pseudo-index. 
It is expected that the intersection number of anti-canonical divisor 
attains to the minimum on an extremal ray, i.e. there exists an extremal rational curve 
whose intersection number with the anti-canonical divisor 
equals the pseudo-index. In this note, we prove this 
for smooth Fano 4-folds having birational contractions.
\end{abstract}

\section{Introduction}
Let $X$ be a Fano manifold, i.e. a smooth projective variety such that the 
anti-canonical divisor $-K_{X}$ is ample. 
Throughout this paper, algebraic varieties are defined over 
the field of complex numbers.

The {\it index} $r(X)$ and the {\it pseudo-index} $i(X)$ of a Fano manifold $X$ is defined 
respectively as
$$
r(X):=\max \{ m\in \N \mid -K_{X}=mH \ \mbox{for some }H\in\Pic (X) \}, \\ 
$$
$$
i (X):=\min \{ -K_{X}\cdot \Gamma \mid \Gamma \mbox{ is a rational curve on } X \}. 
$$
By definition, the positive integer $i(X)$ is a multiple of $r(X)$. 
The equality $r(X)=i(X)$ does not hold in general. 
For example, if $X=\PP^{a}\times\PP^{b}$, 
then we have $r(X)=$gcd$(a+1,b+1)$ while $i(X)=\min (a+1,b+1)$.  
On the other hand,  when $\rho(X)=1$, we do not know whether 
the equality holds or not (see \cite{KollarBook} p. 248 Problem 1.13). 

In \cite{WisManu}, J. A. Wi\'sniewski observed that the pseudo-index is well adapted to 
the study of Fano manifolds with Picard number greater than or equal to 2,  
since it is used to give a lower bound for the dimension of 
the deformation space of rational curves. However,  
in view of the fact that the geometric structure of 
Fano manifolds is governed by its extremal rays, 
it is essential to consider not all rational curves 
but only the extremal rational curves. So, we define another 
invariant $\ell (X)$  as follows.

Recall that the {\it length} of an extremal ray $R\subset \NNE(X)$ 
is defined by   
$$
\ell (R):=\min \{ -K_{X}\cdot \Gamma\ |\ [\Gamma]\in R \}. 
$$   
Note that the Kleiman-Mori cone of a Fano manifold 
is generated by a finite number of extremal rays.  
We define the {\it minimal length} of extremal rays for  
a Fano manifold $X$ as 
$$
\ell (X):=\min \{\,  \ell (R) \mid R \mbox{ is an extremal ray of } \NNE(X) \} 
$$
That is, the positive integer $\ell (X)$ is the minimal anti-canonical degree among 
all extremal rational curves on $X$. Clearly, we have $\ell (X) \geq i(X)$. 
A natural problem is the following:

\

{\bf Problem}: Do we have $i(X)=\ell (X)$ for any Fano manifold $X$ ?

\

The purpose of this note is to give an affirmative answer 
in a special case:

\begin{theo}\label{main}
Let $X$ be a smooth Fano 4-fold. 
Assume that $X$ has a birational contraction. 
Then, we have $i(X)=\ell (X)$.
\end{theo}

Note that in some cases, the equality $i(X)=\ell (X)$ is easily 
verified:
\begin{itemize} 
\item When $\rho(X)=1$, the equality is obvious, since (the numerical class of) any curve on $X$ generates the extremal ray.  
\item If there is an extremal ray $R\subset \NNE(X)$ such that 
$\ell (R)=1$, then clearly $i(X)=\ell (X)=1$.
\item If $X$ is a toric Fano manifold, the equality follows from the fact that 
any curve on $X$ is numerically equivalent to a linear combination of T-invariant curves with natural number coefficients (see the proof of Proposition 2.26 in \cite{Oda}).
\end{itemize}
{\it Remark.} Concerning the last observation on the toric case, 
a similar statement is expected in general.  
For simplicity, we consider a Fano manifold $X$ with $\rho(X)=2$. 
Then the Kleiman-Mori cone is generated by two extremal rays: 
$$
\overline{\NE}(X)=R_{1}+R_{2}.
$$
Let $f_{i}$ be the minimal rational curve of the extremal ray $R_{i}$, 
i.e. we assume that $-K_{X}\cdot f_{i}=\ell (R_{i})$ for $i=1,2$. 

\

{\bf Question}: Let $\Gamma$ be an irreducible curve on $X$. 
Do there exist positive {\it integers} $a_{1}$ and  $a_{2}$ such that 
the 1-cycle $a_{1}f_{1}+a_{2}f_{2}$ is numerically equivalent to 
$\Gamma$ ?

\

The affirmative answer gives the equality $i(X)=\ell (X)$. Indeed, 
if $\Gamma_{0}\subset X$ is a rational curve such that 
$-K_{X}\cdot\Gamma_{0}=i(X)$, we write 
$\Gamma_{0}\equiv a_{1}f_{1}+a_{2}f_{2}$ with 
$a_{1},a_{2}\in \N$, and we get 
$$
i(X)=-K_{X}\cdot \Gamma_{0}=
a_{1}(-K_{X}\cdot f_{1})+a_{2}(-K_{X}\cdot f_{2})\geq 
\min \{ \ell(R_{1}), \ell(R_{2}) \}=\ell (X).
$$ 
In dimension three, the answer to the question is affirmative by \cite{MM} Proposition 6. 
The proof depends on numerical arguments on the intersection numbers of divisors on 3-folds, and seems difficult to apply it to higher dimensions.    
In this note, we treat only the problem of the equality $i(X)=\ell (X)$ 
in a direct way using classification results of Fano 4-folds.

\

The present note is organized as follows. Section \ref{ratcurve} is devoted 
to show a preliminary result based on the bend-and-break lemma. 
In Section \ref{divtopt}, we prove the equality $i(X)=\ell (X)$ 
when $X$ has a birational contraction sending a divisor to a point. 
Section \ref{classification} gives a partial classification of Fano manifolds with $\ell (X)\geq 2$, which is necessary to our purpose. 
The proof of Theorem \ref{main} is done in Section \ref{proof} using 
the results of Section \ref{divtopt} and \ref{classification}.

\

 {\it Notation and conventions.}
The blow-up of a variety $Y$ along a subvariety $C$ is denoted by $\Bl_{C}(Y)$.
 We denote by $Q_{k}$ a smooth hyperquadric in $\PP^{k+1}$. 
For a Cartier divisor $E$ on a variety $X$ and an extremal ray $R\subset \overline{\NE}(X)$, 
the notation $E\cdot R>0$
means that $E\cdot \alpha>0$ 
for some $\alpha\in R$ (hence for any $\alpha\in R\setminus\{ 0 \}$). 
For a vector bundle $\EE$, 
we denote $\PP(\EE)$ the Grothendieck's projectivization.

\section{Unsplit family of rational curves}\label{ratcurve}
For the classification of Fano manifolds, it is important to compute the 
intersection number of extremal rational curves with special divisors.  
In this section, we prove a proposition on the intersection of rational curves 
with the exceptional divisor of a divisorial contraction. 
 
The following lemma is well known but we include the proof for the reader's convenience.

\begin{lem}\label{bendandbreak}
Let $q:S\to B$ be a ruled surface over an irreducible curve $B$. 
Assume that there exists a morphism $p:S\to S'$ such that 
$\dim S'=2$. Let $D$ be an effective divisor on $S$. 
If $\dim p (\Supp (D))=0$, then $\Supp (D)$ is a section 
of $q$.
\end{lem}
\begin{proof} (see \cite{Mori} p.599, \cite{I} p.460, \cite{WisManu} p.138, 
or \cite{KollarBook} Ch.II. 5) We may assume that $B$ is smooth (if $B$ is singular, we consider its normalization $\wtilde{B}\to B$ and the fiber product 
$\wtilde{S}:=S\times _{B}\wtilde{B}$). 
Following the notation of \cite{Hartshorne} Ch.V. Proposition 2.8, 
let $C_{0}$ be a section of $q$ such that $C_{0}^{2}=-e$ and let $f$ be a fiber of $q$.  

{\it Step1}. We show that $D$ is irreducible. If not, let $A_{1}$ and $A_{2}$ be  
distinct irreducible components of $\Supp (D)$. Since $A_{i}$ is  an 
exceptional curve, we have $A_{i}^{2}<0$ ($i=1,2$). 
Since $A_{1}\neq A_{2}$, we have $A_{1}\cdot A_{2}\geq 0$. We write
$$
A_{i}\equiv a_{i}C_{0}+b_{i}f \ \ (i=1,2).
$$
Note that $a_{i}=A_{i}\cdot f\geq 1$ ($i=1,2$). We have 
$a_{2}A_{1}-a_{1}A_{2}\equiv (a_{2}b_{1}-a_{1}b_{2})f$, 
thus $(a_{2}A_{1}-a_{1}A_{2})^{2}=0$. Hence
$$
0\leq 2a_{1}a_{2}(A_{1}\cdot A_{2})=a_{2}^{2}A_{1}^{2}+a_{1}^{2}A_{2}^{2}<0,
$$
a contradiction. We conclude that $D$ is irreducible.

{\it Step 2}. We show that $\Gamma:=\Supp (D)$ is a section. We write 
$\Gamma\equiv aC_{0}+bf$. We consider the case $C_{0}^{2}\leq 0$.
We have $(\Gamma -aC_{0})^{2}=(bf)^{2}=0$. It follows that 
$$
2a (\Gamma\cdot C_{0})=\Gamma^{2}+a^{2} C_{0}^{2}<0.
$$
Hence, $\Gamma \cdot C_{0}<0$, which implies that $\Gamma=C_{0}$ 
is a section. Now, consider the case $C_{0}^{2}> 0$. 
Assume that 
$a=\Gamma\cdot f\geq 2$. Then, by \cite{Hartshorne} Ch.V. Proposition 2.21 (a), 
we have $2b-ae\geq 0$. Hence,  
$$
0\leq a(2b-ae)=(aC_{0}+bf)^{2}=\Gamma^{2}<0, 
$$
a contradiction. Therefore, $\Gamma\cdot f=1$.
\end{proof}

We recall some notation on the family of rational curves 
from \cite{KollarBook} to which we refer the reader for details. 
A {\it family of rational curves} on a projective variety $X$ is 
an irreducible component $V$ of the scheme $\Ratcurve(X)$ 
parameterizing rational curves on $X$. If $V$ is proper,  
it is called {\it unsplit}. Let $U$ be the universal family over $V$. 
Then we have the following basic diagram: 
\begin{equation}\label{diagram}
\begin{CD}
U  @>p >> X \\
@VqVV \\
V
\end{CD}
\end{equation}
where $q$ is a $\PP^{1}$-bundle and $p$ is the map induced by the evaluation map.
For $v\in V$, we denote by $f_{v}$ the corresponding rational curve, i.e. 
$f_{v}:=p(q^{-1}(v))$.

\begin{prop}\label{Ef=1}
Let $\pi: X\to Y$ be the blow-up of a smooth projective variety 
$Y$ of dimension $\geq 3$ along a smooth curve $C$. 
Let $E$ be the exceptional divisor. 
Let $V$ be an unsplit family of rational curves on $X$ 
such that $E\cdot f>0$ for some (hence for any) $[f]\in V$.  
If $\dim V\geq 3$, 
then we have $\sharp (E\cap f)=1$ for any $[f]\in V$ such that 
$f\not\subset E$.
\end{prop}
\begin{proof}
Consider the above diagram (\ref{diagram}) of the family $V$.
For a point $c\in C$ such that $p(U)\cap E_{c}\neq \varnothing$, we put:
$$
E_{c}:=\pi^{-1}(c),\ U_{c}:=p^{-1}(E_{c}),\ V_{c}:=q(U_{c}).
$$
From the assumption that $E\cdot f>0$, the rational curve $f$ is not contracted by $\pi$.  
Thus, $q |_{U_{c}}: U_{c}\to V_{c}$ is a finite map. In particular, 
$\dim U_{c}=\dim V_{c}$. Hence, we have 
\begin{eqnarray*}
\dim V_{c}&=&\dim U_{c} \\ 
&\geq& \dim U-\dim p(U)+\dim (p(U)\cap E_{c}) \\  
&\geq& \dim U-\dim p(U)+(\dim p(U)+\dim E_{c}-\dim X) \\ 
&=&\dim U-2.
\end{eqnarray*}
Since $q: U\to V$ is a $\PP^{1}$-bundle, we have $\dim U=\dim V+1$. 
Therefore, 
\begin{equation}\label{dimension}
\dim V_{c}\geq \dim V-1.
\end{equation}

{\it Step 1.} We first show that if $f\not\subset E$ then 
$\sharp (\pi (f)\cap C)=1$. Assume to the contrary 
that there exists $[f_{0}]\in V$ such that 
$f_{0}\not\subset E$ and $\sharp (\pi (f_{0})\cap C)\geq 2$. 
Let $a, b\in \pi (f_{0})\cap C$ be distinct points. 
Note that $V_{a}\cap V_{b}\neq \varnothing$ 
since the point $[f_{0}]$ lies on the intersection. Using the inequality
 (\ref{dimension}), we have 
$$
\dim (V_{a}\cap V_{b})\geq \dim V_{a}+\dim V_{b}-\dim V
=\dim V -2\geq 1.
$$
Hence, there exists an irreducible curve $B\subset V_{a}\cap V_{b}$. 
Consider the ruled surface $S:=q^{-1}(B)$.  
Since $\pi (f_{0})\not \subset C$, the image $\pi\circ p(S)$ is a surface. 
We see that $U_{a}\cap S$ and $U_{b}\cap S$ are exceptional curves 
because these are respectively contracted to the points $a$ and $b$. 
Thus, we have a contradiction by Lemma \ref{bendandbreak}. 

{\it Step 2.} Consider a rational curve $f$ from the family $V$. 
Assume $f\not \subset E$. By Step 1, we have 
$\sharp (\pi (f)\cap C)=1$. We put $c:=\pi (f)\cap C$. 
By the inequality (\ref{dimension}), there exists an irreducible curve $B$ in $V_{c}$ 
passing through the point $[f]\in V$. Consider the ruled surface 
$S:=q^{-1}(B)$. Since $f\not \subset E$, we see that $p(S)\not \subset E$. 
We write 
$$
p^{*}E |_{S}=D+F
$$ 
where $D$ is the horizontal part and $F$ is the vertical part, i.e. 
$\dim q(D)=1$ and $\dim q(F)=0$. We put $D':=\Supp (D)$ and 
$F':=\Supp (F)$. We have $D'\subset p^{-1}(E)\cap S$, hence $p(D')\subset E$. 
Recall that $B\subset V_{c}$. For any $v\in B\setminus q(F')$, 
we have $f_{v}\cap E_{c}\neq \varnothing$, and hence by Step 1, 
we see that $\pi(f_{v})\cap C=c$, i.e. $f_{v}\cap E\subset E_{c}$. 
Let $u$ be a point 
in $D'\setminus F'$. Since $q(u)\in B\setminus q(F')$, we have 
$f_{q(u)}\cap E\subset E_{c}$. Therefore, 
$$
p(u)\in f_{q(u)}\cap p(D')\subset f_{q(u)}\cap E\subset E_{c}.
$$
Thus, $p(D'\setminus F')\subset E_{c}$. Taking the closure, we conclude that 
$p(D')\subset E_{c}$. Hence, $\pi\circ p(D')=c$. As in Step 1, we see that 
$\pi\circ p(S)$ is a surface. By Lemma \ref{bendandbreak}, $D'$ is a 
section of $q |_{S}$, which implies that $\sharp (E\cap f)=1$. 
\end{proof}

\section{Case of Fano manifolds with a divisorial contraction to a point}\label{divtopt}

We first give an example in which the equality $i(X)=\ell (X)$ is easily verified.
Let $\pi : X\to \PP^{n}$ be the 
blow-up at a point $a$. We assume $n\geq 3$. We consider the diagram: 
$$
\begin{CD}
X@>\pi >> \PP^{n} \\
@V\varphi VV \\
\PP^{n-1}
\end{CD}
$$
where $\varphi :X\to \PP^{n-1}$ is the $\PP^{1}$-bundle 
whose fibers are the strict transforms of lines passing through $a$. 
Let $e$ be a line in the exceptional divisor 
$E\simeq \PP^{n-1}$ and let $f$ be a fiber of $\varphi$. 
Then, $R_{1}:=\RR[e]$ and $R_{2}:=\RR[f]$ are extremal rays. Since 
$\rho (X)=2$, we have 
$$\overline{\NE}(X)=R_{1}+R_{2}.$$
Note that $\ell (R_{1})=-K_{X}\cdot e=n-1$ and 
$\ell (R_{2})=-K_{X}\cdot f=2$. Hence, we have 
$\ell(X)=2$. 
Put $H:=\pi^{*}\OO_{\PP^{n}}(1)$ and 
$L:=\varphi^{*}\OO_{\PP^{n-1}}(1)=H-E$.  
Remark that $L$ is the strict transform of a hyperplane containing the 
point $a$. We get 
$$
-K_{X}=\pi^{*}(-K_{\PP^{n}})-(n-1)E=(n+1)H-(n-1)(H-L)=2H+(n-1)L. 
$$
Note that for a curve $\Gamma\subset X$, we have 
$\pi_{*}\Gamma\not\equiv 0$ or $\varphi_{*}\Gamma\not\equiv 0$. 
If $\Gamma_{0}$ is a rational curve such that 
$(-K_{X})\cdot \Gamma_{0}=i(X)$, then we have 
$$
i (X)=(-K_{X})\cdot \Gamma_{0}
=(2H+(n-1)L)\cdot \Gamma_{0}\geq 2=\ell (X). 
$$
It follows that $i(X)=\ell (X)$. 

\

We generalize the above example as follows: 

\begin{prop}\label{divisortopoint} Let $X$ be a Fano manifold 
of dimension $n\geq 3$. 
Assume that there exists a birational extremal contraction 
$\pi :X\to Y$ sending a divisor to a point.
Then, we have $i(X)=\ell (X)$.   
\end{prop}
\begin{proof}
Let $E$ be the exceptional divisor of $\pi$. 
By \cite{Wis} Corollary 1.3, there exists an extremal ray 
$R\subset\NNE(X)$ such that $E\cdot R>0$, and the associated contraction 
$\varphi=\cont_{R}:X\to Z$ is either:  
\begin{enumerate}
\item a $\PP^{1}$-bundle,
\item a conic bundle with singular fibers, or 
\item a smooth blow-up along a smooth subvariety of codimension 2. 
\end{enumerate}
In case (2) and (3), $\ell (R)=1$. Hence, $i(X)=\ell (X)=1$, and  we are done. 
We assume that $\varphi:X\to Z$ is a $\PP^{1}$-bundle.  
Note that the base space $Z$ is smooth (see \cite{Ando}).
Let $f$ be a fiber of $\varphi$. 
Let $B$ be an irreducible curve on $Z$ passing through the point $\varphi(f)$. 
Consider the ruled surface $S:=\varphi^{-1}(B)$. Since $S\cap E$ is an exceptional curve, 
by Lemma \ref{bendandbreak} it is a section of $\varphi |_{S}$. 
This implies that $E\cdot f=1$ because the exceptional diviosr $E$ 
is reduced (see \cite{KawamataMM} Proposition 5-1-6). 
We conclude that  
$\varphi |_{E}:E\to Z$ is an isomorphism. 
Note that $E\simeq Z$ is smooth.

Using the rank 2 vector bundle 
$\EE:=\varphi_{*}\OO_{X}(E)$, we write $X=\PP(\EE)$. 
Pushing down the exact sequence 
$$
0\to \OO_{X}\to \OO_{X}(E)\to \OO_{E}(E)\to 0,
$$
we obtain
$$
0\to \varphi_{*}\OO_{X}\to \varphi_{*}\OO_{X}(E)\to \varphi_{*}\OO_{E}(E)\to R^{1}\varphi_{*}\OO_{X}=0. 
$$
Here, we have 
$\varphi_{*}\OO_{X}\simeq \OO_{Z}$ and $\varphi_{*}\OO_{E}(E)\simeq \varphi_{*}N_{E/X}$. 
Thus, 
$\det \EE \simeq \varphi_{*}N_{E/X}$. Note also that  
the hyperplane bundle $\OO_{\PP(\EE)}(1)$ is isomorphic to $E$. 
Using 
the canonical bundle formula for the $\PP^{1}$-bundle, we get 
\begin{equation}\label{formula}
K_{X}=-2E+\varphi^{*}(K_{Z}+\varphi_{*}N_{E/X}). 
\end{equation}
Now, assume to the contrary that $\ell (X)>i (X)$, 
i.e. $\ell (X)=2$ and $i (X)=1$. 
Let $\Gamma_{0}\subset X$ be a rational curve such that 
$-K_{X}\cdot\Gamma_{0}=i (X)=1$. In particular,  
$\Gamma_{0}$ is not a fiber of $\varphi$. 
Hence, $\varphi_{*}\Gamma_{0}\not\equiv 0$.
If $\Gamma_{0}\subset E$, then $-K_{X}\cdot \Gamma_{0}\geq \ell (R)\geq \ell (X)=2$, 
a contradiction. Thus, $\Gamma_{0}\not \subset E$ so that $E\cdot \Gamma_{0}\geq 0$. 
Since $\varphi: X\to Z$ is a $\PP^{1}$-bundle, by 
\cite{SW} Theorem 1.6 or \cite{KMM} Corollary 2.9, 
we conclude that $Z$ is a Fano manifold. In particular, 
$-K_{Z}\cdot \varphi_{*}\Gamma_{0}>0$.
Note that the conormal bundle 
$N_{E/X}^{*}$ is ample, since $E$ is an exceptional divisor. 
Therefore, using (\ref{formula}) we have  
$$
1=-K_{X}\cdot \Gamma_{0}=2E\cdot \Gamma_{0}+\varphi^{*}(-K_{Z})\cdot \Gamma_{0}+\varphi^{*}(\varphi_{*}N_{E/X}^{*})\cdot \Gamma_{0}
\geq 0+1+1=2, 
$$
a contradiction.  \end{proof}

\section{Classification results}\label{classification}
In this section, we present results on a partial classification of 
Fano manifolds with $\ell (X)\geq 2$. These are 
used in the next section to prove our Theorem \ref{main}.
 
\begin{lem}\label{ER>0} Let $Y$ be a smooth projective variety of 
dimension $n\geq 4$. Let $\pi:X\to Y$ be the blow-up 
along a smooth curve $C\subset Y$. Assume that $X$ is 
a Fano manifold. Let $E$ be the exceptional divisor of $\pi$. 
Then, there exists an extremal ray $R\subset \NNE(X)$ such that $E\cdot R>0$. 
Furthermore, every non-trivial fiber of 
the associated contraction $\varphi:X\to Z$ has dimension at most 2. 
\end{lem}
\begin{proof}
This follows from a similar argument as in \cite{BCW} Section 2.
\end{proof}
Throughout the section, we fix the notation of this lemma.
\begin{prop}\label{CR}
If $\varphi$ is a fiber type contraction with $\dim Z=n-2$ and $\ell (R)\geq 2$, 
then we have either: $Y=\PP^{n}$ and 
$C$ is a line, 
or $Y=Q_{n}$ and $C$ is a conic not on a plane contained in 
$Q_{n}$.
\end{prop}
\begin{proof} Since $\ell (R)\geq 2$,  the general fiber of $\varphi$ is 
isomorphic to $\PP^{2}$ or $Q_{2}$. 
Hence, the statement follows from \cite{T1} Theorem 1.1.
\end{proof}
If the general fiber of $\varphi$ has dimension one, 
we have $\ell (R)\leq 2$. We prove the following:  
\begin{prop}\label{p1}
If $\varphi$ is a fiber type contraction with $\dim Z=n-1$ and $\ell (R)=2$, 
then the pair $(Y,C)$ is exactly one of the following: 
\begin{enumerate}
\item $Y=Q_{n}$ and $C$ is a line;
\item $Y=\PP^{1}\times \PP^{n-1}$ and $C$ is a fiber of 
the projection $\PP^{1}\times \PP^{n-1}\to \PP^{n-1}$;
\item $Y=\Bl_{\PP^{n-2}}(\PP^{n})$ and $C$ is the strict transform 
of a line in $\PP^{n}$;
\item $Y=\Bl_{\PP^{n-2}}(\PP^{n})$ and $C$ is a fiber of 
the exceptional divisor;
\item $Y=\PP(\OO_{\PP^{1}}\oplus\OO_{\PP^{1}}(1)^{\oplus(n-1)})$ and 
$C$ is the section such that 
$N_{C/Y}\simeq \OO_{\PP^{1}}(-1)^{\oplus(n-1)}$.
\end{enumerate}
\end{prop}
\begin{proof}
Let $f\simeq \PP^{1}$ be a general fiber of $\varphi$ 
such that $f\not\subset E$. Since 
$\ell (R)=2$, the family of rational curves containing the point $[f]$ is unsplit.
By Proposition \ref{Ef=1}, we have $\sharp(E\cap f)=1$. 
This implies $E\cdot f=1$ because the exceptional divisor $E$ is reduced.

We first consider the case where 
the restriction map $\varphi|_{E}:E\to Z$ is not finite, i.e. there exists a 
curve $\wtilde{C}$ contained in $E$ and contracted by 
$\varphi$. By definition of an extremal contraction, there 
exists $b\in\RR$ such that $\wtilde{C}\equiv bf$. 
Recall that $E$ is a $\PP^{n-2}$-bundle over the curve $C$. 
Since $\wtilde{C}$ is an exceptional curve, its numerical class 
generates an extremal ray. Since $\rho(E)=2$, we have  
$$
\NNE(E)=\RR[\wtilde{C}]+\RR[e]
$$ 
where $e$ is a line in a fiber of $\pi |_{E}:E\to C$.
Using the adjunction formula: $K_{E}=(K_{X}+E)|_{E}$ and the 
equality $E\cdot f=1$, we get
$$
-K_{E}\cdot e=-K_{X}\cdot e-E\cdot e=(n-2)+1=n-1>0
$$
and 
\begin{equation}\label{b}
-K_{E}\cdot \wtilde{C}=b(-K_{X}\cdot f-E\cdot f)=b(2-1)=b>0.
\end{equation}
By Kleiman's criterion, $-K_{E}$ is ample. Since the Fano manifold $E$ 
is a $\PP^{n-2}$-bundle over a curve and contains an exceptional curve, it is isomorphic to 
$\Bl_{\PP^{n-3}}(\PP^{n-1})$. 
Since $E\cdot f=1$, we see that $\varphi |_{E}:E\to Z$ is generically one to one 
onto the normal variety $Z$. Hence, the finite part of its Stein factorization is an isomorphism. 
It follows that $\varphi |_{E}:E\to Z$ coincides with the blow-up
$\Bl_{\PP^{n-3}}(\PP^{n-1})\to\PP^{n-1}$. Hence, 
$Z$ is isomorphic to $\PP^{n-1}$. We have also $\rho(X)=\rho(Z)+1=2$ and $\rho(Y)=1$.  
We observe that $\varphi_{*}e$ is a line in $Z\simeq \PP^{n-1}$. 
So, if we put $L:=\varphi^{*}\OO_{Z}(1)$, then $L\cdot e=1$. 
Since $-K_{X}\cdot e=n-2$ and $-K_{X}\cdot f=2$, we have 
$-K_{X}= nL+2E$. On the other hand, 
$$
-K_{X}=\pi^{*}(-K_{Y})-(n-2)E=r(Y)H-(n-2)E
$$
where $r(Y)$ is the index of $Y$ and $H$ is the pull back by $\pi$ 
of the ample generator of $\Pic(Y)\simeq \Z$. Thus, we have 
$$
r(Y)H=n(L+E).
$$
Since $(L+E)\cdot e=0$, there exists $D\in \Pic(Y)$ such that 
$L+E=\pi^{*}D$.  Note that we can write $\pi^{*}D=dH$ with 
$d\in \Z$. Hence, $r(Y)=nd$. 
By \cite{KO}, we have $r(Y)=n$ and $Y$ is isomorphic to $Q_{n}$. 
We have also $H\cdot f=1$. Now, we know that the curve $\wtilde{C}$ 
defined above is a fiber of the exceptional divisor of the blow-up 
$E\simeq \Bl_{\PP^{n-3}}(\PP^{n-1})\to\PP^{n-1}$. 
Since $-K_{E}\cdot \wtilde{C}=1$, we get $b=1$ from (\ref{b}), and 
we see that $\wtilde{C}$ is numerically equivalent to $f$. Hence, we have 
$$
\OO_{Y}(1)\cdot C=H\cdot \wtilde{C}=H\cdot f=1.
$$
It follows that $C$ is a line in $Y\simeq Q_{n}$. 
Hence, we get the example (1). 

Now, we consider the case where $\varphi |_{E}: E\to Z$ is finite. 
Note that every fiber of $\varphi$ is one-dimensional. 
Hence, by the assumption that $\ell (R)=2$,  
we see that $\varphi$ is a $\PP^{1}$-bundle 
(see \cite{Ando} and \cite{Wis} Theorem 1.2). 
By \cite{SW} Theorem 1.6 or \cite{KMM} Corollary 2.9, we conclude that 
$Z$ is a Fano manifold. 
Note that $\varphi |_{E}$ is an isomorphism because $E\cdot f=1$.  
Since $Z\simeq E$ has a structure of a 
$\PP^{n-2}$-bundle over the curve $C$, $Z$ is isomorphic to $\PP^{1}\times \PP^{n-2}$ or 
$\Bl_{\PP^{n-3}}(\PP^{n-1})$, and $C$ is isomorphic to $\PP^{1}$. 
\begin{claim}
$Y$ is a $\PP^{n-1}$-bundle over $\PP^{1}$.
\end{claim}
\begin{proof}
For a point $a\in C$, we put $E_{a}:=\pi^{-1}(a)$,  
$Z_{a}:=\varphi(E_{a})$,  
$X_{a}:=\varphi^{-1}(Z_{a})$ and $Y_{a}:=\pi (X_{a})$. 
Note that $X_{a}$ is smooth because it is a $\PP^{1}$-bundle over $Z_{a}\simeq\PP^{n-2}$. 
Hence, the divisor $Y_{a}\subset Y$ is smooth in codimension one.  
Thus, $Y_{a}$ is normal (\cite{Hartshorne} Ch. II Proposition 8.23).  
Note that $N_{E_{a}/X_{a}}\simeq \OO_{\PP^{n-2}}(-1)$. 
It follows that $\pi |_{X_{a}}: X_{a}\to Y_{a}$ is the blow-up 
at the point $C\cap Y_{a}$ with the exceptional divisor 
$E_{a}$. On the other hand, 
$\varphi |_{X_{a}}:X_{a}\to Z_{a}\simeq \PP^{n-2}$ is
a $\PP^{1}$-bundle. 
Consider the composite map $\Phi: X\to Z\simeq E\to C\simeq \PP^{1}$.
The fiber $X_{a}=\Phi^{-1}(a)$ is a Fano manifold. 
So, by the classification result due to \cite{BCW}, 
we conclude that $Y_{a}$ is isomorphic to $\PP^{n-1}$.
Consider the nef divisor $F:=\Phi^{*}\OO_{\PP^{1}}(1)$. 
We see that $F-K_{X}$ is ample. Hence, $H^{1}(X,\OO_{X}(F))=0$. 
From the exact sequence
$$
0\to \OO_{X}(F)\to \OO_{X}(F+E)\to \OO_{E}(F+E)\to 0, 
$$
we get 
$$
h^{0}(X,\OO_{X}(F+E))\geq h^{0}(X, \OO_{X}(F))
=h^{0}(\PP^{1},\OO_{\PP^{1}}(1))=2.
$$
Let $M$ be a general member of $|F+E|$ and we put $M':=\pi(M)$. 
Since $M\cdot f=(F+E)\cdot f=E\cdot f=1$, 
we have $M' |_{Y_{a}}\simeq \OO_{\PP^{n-1}}(1)$. 
Hence the morphism $\Phi\circ \pi^{-1}: Y\to \PP^{1}$ 
is a $\PP^{n-1}$-bundle and $C$ is a section. 
\end{proof}
If $Y$ is a Fano manifold, $Y$ is isomorphic to 
$\PP^{1}\times\PP^{n-1}$ or
$\Bl_{\PP^{n-2}}(\PP^{n})$. 
We first treat the case $Y\simeq \PP^{1}\times\PP^{n-1}$. 
Assume that $C$ is not a fiber of the projection 
$pr: \PP^{1}\times \PP^{n-2}\to \PP^{n-2}$. 
Let $\Gamma$ be a fiber of $pr$ meeting $C$ and 
$\wtilde{\Gamma}$ its strict transform by the blow-up $\pi: X\to Y$. 
Then, we have 
$$
K_{X}\cdot \wtilde{\Gamma}
=K_{Y}\cdot \Gamma+(n-2)E\cdot \wtilde{\Gamma}
\geq -2+(n-2)=n-4\geq 0, 
$$
which contradicts the assumption that $X$ is a Fano manifold.
It follows that $C$ is a fiber of the projection $pr$ and we get the 
example (2). 
Now, we consider the case $Y\simeq \Bl_{\PP^{n-2}}(\PP^{n})$. 
Let $G$ be the exceptional divisor of the blow-up 
$\beta: \Bl_{\PP^{n-2}}(\PP^{n})\to \PP^{n}$. 
Assume that $G\cdot C>0$. Let $g$ be a fiber of the $\PP^{1}$-bundle $G\to \PP^{n-2}$ such that 
$g\cap C\neq \varnothing$.
Then, we have 
$$
K_{X}\cdot \wtilde{g}=K_{Y}\cdot g+(n-2)E\cdot 
\wtilde{g}\geq -1+(n-2)=n-3> 0, 
$$
a contradiction. Hence, $G\cdot C\leq 0$. Since $C$ is a section of the 
$\PP^{n-1}$-bundle $Y\to \PP^{1}$, 
$C$ is either the strict transform by $\beta$ of a line in $\PP^{n}$
which does not meet the center $\PP^{n-2}$, 
or a fiber of the $\PP^{1}$-bundle $G\to \PP^{n-2}$. 
Thus, we get the examples (3) and (4).
 
If $Y$ is not a Fano manifold, by  
\cite{Wis} Proposition 3.5, we have $N_{C/Y}\simeq \OO_{\PP^{1}}(-1)^{\oplus(n-1)}$.
Hence, we conclude that 
$Y\simeq \PP(\OO_{\PP^{1}}\oplus \OO_{\PP^{1}}(1)^{\oplus(n-1)})$ 
and we get the example (5).
\end{proof}

{\it Remarks.} In dimension three, there is another example: $Y=\PP^{3}$ and $C$ is 
a rational curve of degree 3 (see \cite{MM} $n^{\circ}$ 27 in Table 2).
If we assume $i(X)\geq 2$, a similar statement is derived from \cite{AO} Theorem 1.3. 
In dimension four, if we assume $\rho(X)=2$ and 
$\varphi:X\to Z$ is a {\it scroll in the sense of adjunction theory} 
(see \cite{BS} for the definition), the example $(Y,C)=(Q_{4}, \mbox{line})$ 
is obtained from the list in \cite{Langer}.

\section{Proof of Theorem \ref{main}}\label{proof}

Let $X$ be a smooth Fano 4-fold. 
Let $R\subset \NNE(X)$ be the extremal ray 
defining the birational contraction $\pi:X\to Y$.
If $\ell(R)=1$, then $\ell(X)=i (X)=1$, hence we are done. 
So, it suffices to consider the case $\ell (R)\geq 2$. 
Recall that a contraction $\pi:X\to Y$ is called $(a,b)$-{\it type} 
if $\dim \Exc(\pi)=a$ and $\dim \pi (\Exc(\pi))=b$. 
By Fujita-Ionescu-Wi\'sniewski's inequality (see \cite{I} Theorem 0.4 and \cite{Wis} Theorem 1.1): 
$$
\dim (\mbox{non-trivial fiber of }\pi)
+\dim \Exc(\pi)\geq \dim X+\ell (R)-1, 
$$
we conclude that $\pi:X\to Y$ is either of type (3,0) or (3,1). 
By Proposition \ref{divisortopoint}, we have $i (X)=\ell (X)$ 
in the case of (3,0)-type. 

Now, assume that $\pi:X\to Y$ is a (3,1)-type contraction.
Since $\ell (R)\geq 2$, $\pi$ is a blow-up 
along a smooth curve $C$ and $Y$ is smooth (see \cite{Tk}).
Let $E$ be the exceptional divisor of $\pi$. By Lemma \ref{ER>0}, 
there exists an extremal ray $R'\subset \overline{\NE}(X)$ 
such that $E\cdot R'>0$. Let $\varphi:X\to Z$ be 
the associated contraction. Recall that the fiber of $\varphi$ 
has dimension at most 2. Since $\ell (R')\geq \ell (X)\geq 2$, 
$\varphi$ is one of the following:
\begin{enumerate} 
\item (3,1)-type: blow-up along a smooth curve and $Z$ is smooth; 
\item (4,2)-type: the general fiber of $\varphi$ is isomorphic to $\PP^{2}$ or $Q_{2}$; 
\item (4,3)-type: the general fiber of $\varphi$ is isomorphic to $\PP^{1}$ and forms an unsplit family. 
\end{enumerate}
The case (1) is excluded by \cite{T4} Proposition 5.
In the case (2), by Proposition \ref{CR}, we have 
$Y=\PP^{4}$ and $C$ is a line, or $Y=Q_4$ and $C$ is a conic 
(not on a plane contained in $Q_4$). 
Concerning the case (3), note that among 5 examples 
in Proposition \ref{p1}, the condition $\ell (X)\geq 2$ is satisfied 
only for the examples (1) and (2).
Summing up, we have the following possibilities:
\begin{enumerate}
\item $Y=\PP^{4}$ and $C$ is a line;
\item $Y=Q_{4}$ and $C$ is a conic;
\item $Y=Q_{4}$ and $C$ is a line;
\item $Y=\PP^{1}\times \PP^{3}$ and $C$ is a fiber 
of the projection $\PP^{1}\times \PP^{3}\to \PP^{3}$.
\end{enumerate}

Hence, to prove Theorem \ref{main}, it is sufficient 
to verify that the equality $i (X)=\ell (X)$ 
holds for $X=\Bl_{C}(Y)$ in the above four examples.
We just check the case (1), since the argument is similar in the other cases. 
Let $C$ be a line in $\PP^{4}$. 
Note that $X=\Bl_{C}(\PP^{4})$ has two extremal contractions:
the blow-up $\pi: X\to \PP^{4}$ along $C$ and 
the $\PP^{2}$-bundle $\varphi: X\to \PP^{2}$. The 
exceptional divisor $E=\Exc (\pi)$ is isomorphic to 
$\PP^{1}\times \PP^{2}$ and the restrictions $\pi |_{E}:E\to C\simeq \PP^{1}$ 
and $\varphi |_{E}: E\to \PP^{2}$ are the two natural projections. 
Let $e$ be a line in a fiber of $\pi |_{E}$ and $f$ be a fiber of $\varphi |_{E}$. 
Then, we have $\NNE(X)=\RR[e]+\RR[f]$. 
Note that $f$ is numerically equivalent to the strict transform by $\pi$ of 
a line meeting $C$. Since $-K_{X}\cdot e=2$ and 
$-K_{X}\cdot f=3$, we have $\ell (X)=2$.  
Let $\Gamma$ be an irreducible curve on $X$. 
Assume that $\Gamma\not\subset E$. Let $H$ be 
the pull back by $\pi$ of a hyperplane containing the line $C$ but 
not containing the curve $\pi (\Gamma)$. Then, we have 
$H\cdot \Gamma\geq E\cdot \Gamma$. Hence, 
$$
-K_{X}\cdot \Gamma =(\pi^{*}(-K_{\PP^{4}})-2E)\cdot \Gamma
=(5H-2E)\cdot \Gamma\geq 3H\cdot \Gamma\geq 3.
$$
If $\Gamma\subset E\simeq \PP^{1}\times\PP^{2}$, 
there exists natural numbers $a$ and $b$ such that $\Gamma\equiv 
ae+bf$. We have $-K_{E}\cdot \Gamma=3a+2b$, and 
$E |_{E}\cdot \Gamma=aE\cdot e+bE\cdot f=-a+b$. Hence, 
$$
-K_{X}\cdot \Gamma=(-K_{E}+E |_{E})\cdot \Gamma=(3a+2b)+(-a+b)=2a+3b\geq 2.
$$
Therefore, $i (X)\geq 2$. 
Thus, $i (X) =\ell (X)=2$. 
\qed
  
\
  
{\it Acknowledgements.}\  
The author would like to thank Professors Gianluca Occhetta and 
Hiroshi Sato for helpful comments.

\noindent -----------------------------------------

\noindent {\small Toru TSUKIOKA \ \ \ \ e-mail:\ tsukioka@las.osakafu-u.ac.jp\\ 
Faculty of Liberal arts and Sciences, 
Osaka Prefecture University \\
1-1\ Gakuen-cho\ Nakaku\ Sakai, 
Osaka 599-8531\ Japan
}

\end{document}